\documentclass{gtpart}

\usepackage{amsmath}
\usepackage{amsthm}
\usepackage{amsopn}
\usepackage{amssymb}
\usepackage[all]{xy}

\title{The Goodwillie tower for $S^1$ and Kuhn's theorem}
\author{Mark Behrens}
\givenname{Mark}
\surname{Behrens}
\address{Department of Mathematics \\ Massachusetts Institute of Technology \\ \newline
Cambridge, MA  02140 \\ USA}
\email{mbehrens@math.mit.edu}
\urladdr{http://math.mit.edu/~mbehrens}


\newcommand{\mc}[1]{\mathcal{#1}}
\newcommand{\mb}[1]{\mathbb{#1}}
\newcommand{\mr}[1]{\mathrm{#1}}

\newcommand{\mital}[1]{\mathit{#1}}

\newcommand{\abs}[1]{\lvert #1 \rvert}

\newcommand{\ul}[1]{\underline{#1}}

\newcommand{\td}[1]{\widetilde{#1}}

\newcommand{\ZZ}{\mathbb{Z}}
\newcommand{\RR}{\mathbb{R}}

\newcommand{\FF}{\mathbb{F}}

\newcommand{\PP}{\mathbb{P}}
\newcommand{\DD}{\mathbb{D}}

\newcommand{\Id}{\mathrm{Id}}

 \newtheorem{thm}[equation]{Theorem}
 \newtheorem{cor}[equation]{Corollary}
 \newtheorem{lem}[equation]{Lemma}
 
\newtheorem*{thm*}{Theorem}
\newtheorem*{cor*}{Corollary}
\newtheorem*{lem*}{Lemma}
\newtheorem*{prop*}{Proposition}

\theoremstyle{definition}

 \newtheorem{rmk}[equation]{Remark}

\newtheorem*{defn*}{Definition}
\newtheorem*{ex*}{Example}
\newtheorem*{exs*}{Examples}
\newtheorem*{rmk*}{Remark}
\newtheorem*{claim*}{Claim}

\newtheorem*{Ack}{Acknowledgments}
\numberwithin{equation}{section}
\numberwithin{figure}{section}

\DeclareMathOperator{\im}{Im}

\DeclareMathOperator*{\holim}{holim}

\DeclareMathOperator{\Tot}{Tot}
\DeclareMathOperator{\Top}{Top}
\DeclareMathOperator{\Sp}{Sp}

\begin{document}

\begin{abstract}
We analyze the homological behavior of the attaching maps in the $2$-local Goodwillie tower of the identity evaluated at $S^1$.  We show that they exhibit the same homological behavior as the James-Hopf maps used by N.~Kuhn to prove the $2$-primary Whitehead conjecture.  We use this to prove a calculus form of the Whitehead conjecture: the Whitehead sequence is a contracting homotopy for the Goodwillie tower of $S^1$ at the prime $2$.
\end{abstract}

\maketitle

\section{Introduction and statement of results}

The aim of this paper is to explain the relationship between the Goodwillie tower of the identity evaluated on  $S^1$ and the Whitehead conjecture (proved by N.~Kuhn \cite{Kuhn}).  Such a relationship has been conjectured by Arone, Dwyer, Lesh, Kuhn, and Mahowald (see \cite{AroneLesh}, \cite{AroneDwyerLesh}, and \cite{Behrens}).  

The author has learned that similar theorems to the main theorems of this paper (Theorem~\ref{thm:main} and Corollary~\ref{cor:main}) were proved recently by Arone-Dwyer-Lesh, by very different methods. The two proofs were discovered independently and essentially at the same time.

Throughout this paper we freely use the terminology of Goodwillie's homotopy calculus of functors \cite{Goodwillie} and Weiss's orthogonal calculus \cite{Weiss}. 
We use the notation:
\begin{align*}
\{ P_i(F) \} & = \text{Goodwillie tower of $F$}, \\
D_i(F) & = \text{$i$th layer of the Goodwillie tower}, \\
\DD_i(F) & = \text{infinite delooping of $F$ (a spectrum valued functor)}, \\
\partial_i(F) & = \text{$i$th Goodwillie derivative of $F$ (a $\Sigma_i$-spectrum)}, \\
P_i^W, D_i^W, \DD_i^W & = \text{the corresponding constructions in Weiss calculus}. \\
\end{align*}
When $F = \Id$, we omit it from the notation.  We use $E^\vee$ to denote the Spanier-Whitehead dual of a spectrum $E$.  We let $\mr{conn}(X)$ denote the connectivity of a space $X$.  As usual, we let $QX$ denote the space $\Omega^\infty \Sigma^\infty X$.  All homology and cohomology is implicitly taken with $\FF_2$ coefficients.  \emph{Everything in this paper is implicitly localized at the prime $2$}.

Let $\Sp^{n}(S)$ denote the $n$th symmetric power of the sphere spectrum.  There are natural inclusions $\Sp^n(S) \hookrightarrow \Sp^{n+1}(S)$, and, by the Dold-Thom theorem, the colimit is given by  
$$ \Sp^\infty(S) = \varinjlim \Sp^n(S) \simeq H\ZZ. $$
Thus the symmetric products of the sphere spectrum may be regarded as giving an increasing filtration of the integral Eilenberg-MacLane spectrum. 
Nakaoka \cite{Nakaoka} showed that ($2$-locally) the quotients $\Sp^n(S)/\Sp^{n-1}(S)$ are non-trivial only when $n = 2^k$.  The non-trivial quotients are therefore given by the spectra
$$ L(k) := \Sigma^{-k} \Sp^{2^k}(S)/\Sp^{2^{k-1}}(S). $$
These spectra were studied extensively by Kuhn, Mitchell, and Priddy \cite{KMP}, and occur in the stable splittings of classifying spaces.
Applying $\pi_*$ to the symmetric powers filtration gives rise to an exact couple, and hence a homological type spectral sequence
\begin{equation}\label{eq:WSS}
E^1_{k,t} = \pi_{t}L(k) \Rightarrow \pi_{k+t} H\ZZ.
\end{equation}
Kuhn's theorem \cite{Kuhn}, known as the `Whitehead conjecture,' states that this spectral sequence collapses at $E_2$, where it is concentrated on the $k=0$ line.

Arone and Mahowald \cite{AroneMahowald} proved that the layers of the ($2$-local) Goodwillie tower of the identity functor evaluated on spheres satisfy $D_i(S^n) \simeq \ast$ unless $i = 2^k$.  Arone and Dwyer \cite{AroneDwyer} proved there are equivalences
\begin{equation}\label{eq:AroneDwyer}
\Sigma^k \DD_{2^k}(S^1) \simeq \Sigma L(k).
\end{equation}
Applying $\pi_*$ to the the fiber sequences
$$ \Omega^\infty \DD_{2^k}(S^1) \rightarrow P_{2^k}(S^1) \rightarrow P_{2^{k-1}}(S^1) $$
results in an exact couple, giving the \emph{Goodwillie spectral sequence} for $S^1$:
\begin{equation}\label{eq:GSS}
E_1^{k,t} = \pi_t D_{2^k}(S^1) \Rightarrow \pi_t(S^1).
\end{equation}

Spectral sequences (\ref{eq:WSS}) and (\ref{eq:GSS}) both converge to $\ZZ$, and, by (\ref{eq:AroneDwyer}), have isomorphic $E_1$-terms.  They differ in that one is of homological type, and one is of cohomological type, and thus in particular their $d_1$-differentials go in opposite directions.  Kuhn's theorem leads to the following natural question: does the Goodwillie spectral sequence for $S^1$ also collapse at its $E_2$-page?  More specifically, do the $d_1$-differentials in each of the spectral sequences serve as contracting chain homotopies for the $E_1$-pages of the other?  The aim of this paper is to prove that indeed this is the case.

To more precisely state the main theorem of this paper, we need to recall exactly what Kuhn proved in \cite{Kuhn}. 
In his proof of the $2$-primary Whitehead conjecture, Kuhn formed a Kahn-Priddy sequence
\begin{equation}\label{eq:KahnPriddy}
S^1 \leftrightarrows \Omega^\infty \Sigma L(0) 
\begin{array}{c} d_0 \\ \leftrightarrows \\ \delta_0 \end{array}
\Omega^\infty \Sigma L(1) 
\begin{array}{c} d_1 \\ \leftrightarrows \\ \delta_1 \end{array}
\Omega^\infty \Sigma L(2) 
\begin{array}{c} d_2 \\ \leftrightarrows \\ \delta_2 \end{array}
\cdots
\end{equation}
The maps $d_k$ are the infinite loop space maps induced by the composites
$$ L(k+1) = \Sigma^{-k-1} \mr{Sp}^{2^{k+1}}(S)/\mr{Sp}^{2^{k}}(S) \xrightarrow{\partial} \Sigma^{-k} \mr{Sp}^{2^{k}}(S)/\mr{Sp}^{2^{k-1}}(S) = L(k).  $$
The maps $d_k$ may be regarded as the attaching maps between consecutive layers of the symmetric powers filtration on $H\ZZ$, and induce on $\pi_*$ the $d_1$-differentials in the spectral sequence (\ref{eq:WSS}).
To define the maps $\delta_k$, Kuhn constructed summands $X_k$ of the suspension spectra $\Sigma^\infty (S^1)^{\wedge 2^k}_{h\Sigma_2^{\wr k}}$, and showed that these summands are equivalent to $\Sigma L(k)$ \cite[Cor.~1.7]{Kuhn}.  In particular, we have retractions
\begin{equation}\label{eq:summand} 
\xymatrix@C-2em{
\Sigma L(k) \ar[rr]^{\mital{id}} \ar[dr]_{\iota_k} &&
\Sigma L(k) 
\\
& \Sigma^\infty (S^1)^{\wedge 2^k}_{h\Sigma_2^{\wr k}} \ar[ur]_{p_k}.
}
\end{equation}
The maps $\delta_k$ in (\ref{eq:KahnPriddy})
 are given by the composites
$$ \Omega^\infty \Sigma L(k) \xrightarrow{\Omega^\infty \iota_k} Q (S^1)^{\wedge 2^k}_{h\Sigma_2^{\wr k}} \xrightarrow{JH} Q (S^1)^{\wedge 2^{k+1}}_{h\Sigma_2^{\wr (k+1)}} \xrightarrow{\Omega^\infty p_{k+1}} \Omega^\infty \Sigma L(k+1). $$
Here,  $JH$ is the James-Hopf map.
Kuhn showed that the sum
$$ d_{k}\delta_k+\delta_{k-1}d_{k-1} $$
is a self-equivalence of $\Omega^\infty \Sigma L(k)$.  
This amounts to an analysis of the diagram
\begin{equation}\label{eq:Kuhn}
\xymatrix@C-1em{
& E_{-1} \ar@{=}[dl] \ar@/^1pc/@{.>}[dr] &&
E_{0} \ar[dl] \ar@/^1pc/@{.>}[dr]^{h_0} &&
E_{1} \ar[dl] \ar@/^1pc/@{.>}[dr]^{h_1} &&
\\
S^1 && 
\Omega^\infty \Sigma L(0) \ar[ll] \ar[ul] && 
\Omega^\infty \Sigma L(1) \ar[ll]^{d_0} \ar[ul]^{\td{d}_0} &&
\Omega^\infty \Sigma L(2) \ar[ll]^{d_1} \ar[ul]^{\td{d}_1} &
\cdots
}
\end{equation}
where the infinite loop spaces $E_k$ fit into fiber sequences
$$ E_{k} \rightarrow \Omega^\infty \Sigma L(k) \xrightarrow{\td{d}_{k-1}} E_{k-1} $$
and the maps $h_k$ are given by the composites
$$ E_k \rightarrow \Omega^\infty \Sigma L(k) \xrightarrow{\delta_k} \Omega^\infty \Sigma L(k+1). $$ 
To prove the Whitehead conjecture, Kuhn proved the following theorem.

\begin{thm}[{\cite{Kuhn}}]\label{thm:Kuhn}
The composites $\td{d}_k \circ h_k$ are equivalences.
\end{thm}

Kuhn proved Theorem~\ref{thm:Kuhn} by showing that $\td{d}_k \circ h_k$ is a homology equivalence.

We now turn our attention to the Goodwillie tower of the identity functor.
As explained in \cite{AroneDwyerLesh},
precomposing the Goodwillie tower of the identity with the functor
$$ \chi: V \mapsto S^V $$
gives the Weiss tower for the functor $\chi$: we have natural 
equivalences of towers of functors from vector spaces to spaces
$$ \{ P^W_i(\chi)(V) \} \simeq \{ P_i(\Id)(S^V) \}. $$
(See \cite[Lem.~1.2]{Arone}, \cite[Ex.~5.7]{Weiss} for the proof of an almost identical result.)
Let
$$ \phi_k : D_{2^k}(S^V) \rightarrow  BD_{2^{k+1}}(S^V) $$
be the attaching map between consecutive non-trivial layers in the Weiss tower.  Arone-Dwyer-Lesh prove that there exist natural transformations
$$ \psi_k : B^kD_{2^k}(S^V) \rightarrow  B^{k+1}D_{2^{k+1}}(S^V) $$
so that 
$$ \Omega^k \psi_k = \phi_k. $$

Under the Arone-Dwyer equivalence $\Sigma L(k) \simeq \Sigma^k \DD_{2^k}(S^1)$, we get a delooped calculus version of the Kahn-Priddy sequence
\begin{equation*}
S^1 \leftrightarrows D_1(S^1) 
\begin{array}{c} d_0 \\ \leftrightarrows \\ \psi_0 \end{array}
B D_2(S^1)
\begin{array}{c} d_1 \\ \leftrightarrows \\ \psi_1 \end{array}
B^2 D_4(S^1)  
\begin{array}{c} d_2 \\ \leftrightarrows \\ \psi_2 \end{array}
\cdots
\end{equation*}

Our main theorem is the following (Conjecture 1.4 of \cite{AroneDwyerLesh}).

\begin{thm}\label{thm:main}
The sums
$$ d_k \psi_k + \psi_{k-1} d_{k-1} $$
are equivalences.
\end{thm}

\begin{rmk}
Our proof uses no specific properties of the natural transformations $\psi_k$, except for the fact that they are  $k$-fold deloopings of the natural transformations $\phi_k$.  Therefore Theorem~\ref{thm:main} holds independently of the choice of the deloopings.
\end{rmk}

As the maps $\psi_k$ induced the $d_1$-differentials in the spectral sequence (\ref{eq:GSS}) on $\pi_*$, we get the following corollary.

\begin{cor}\label{cor:main}
The Goodwillie spectral sequence for $S^1$ collapses at the $E_2$ page.
\end{cor}

Our proof of the main theorem is similar to that of Kuhn in that we analyze the diagram of fiber sequences
\begin{equation}\label{eq:mydiag}
\xymatrix@C-1em{
& E_{-1} \ar@{=}[dl] \ar@/^1pc/@{.>}[dr] &&
E_{0} \ar[dl] \ar@/^1pc/@{.>}[dr]^{h'_0} &&
E_{1} \ar[dl] \ar@/^1pc/@{.>}[dr]^{h'_1} &&
\\
S^1 && 
D_1(S^1) \ar[ll] \ar[ul] && 
B D_2(S^1) \ar[ll]^{d_0} \ar[ul]^{\td{d}_0} &&
B^2 D_4(S^1) \ar[ll]^{d_1} \ar[ul]^{\td{d}_1} &
\cdots
}
\end{equation}
where the maps $h'_k$ are given by the composites
$$ E_k \rightarrow B^k D_{2^k}(S^1) \xrightarrow{\psi_k} B^{k+1}D_{2^{k+1}}(S^1). $$ 
To prove Theorem~\ref{thm:main}, it suffices to show that the composites $\td{d}_k \circ h'_k$ are equivalences. 
We prove this by establishing that these composites induce isomorphisms on mod $2$ homology.  We will do this by endowing $H_* \Omega^\infty \Sigma L(k)$ with a weight filtration, and will prove

\begin{thm}\label{thm:psidelta}
The induced maps 
$$ (\psi_k)_*, (\delta_k)_*: E_0 H_*\Sigma L(k) \rightarrow E_0 H_* \Sigma L(k+1) $$
on the associated graded homology groups with respect to the weight filtration are equal.
\end{thm}

This theorem, together with the observation that the maps $d_k$ behave well with respect to the weight filtration, will allow us to deduce Theorem~\ref{thm:main}.

The homological analysis of the maps $\psi_k$ will be performed by observing that, at least up to the weight filtration, the homological behavior of the attaching maps between $i$th and $2i$th layers of any functor $F$ from spaces to spaces is essentially dictated by the homological behavior of the left action of the operad $\partial_*(\Id)$ on the derivatives $\partial_*(F)$ (such left operadic module structure exists by the work of Arone and Ching \cite{AroneChing}, \cite{Ching}).  The $2$-primary homological behavior of this action in the case of the identity functor, when evaluated on spheres, was determined by the author \cite{Behrens}.

This paper is organized as follows.  In Section~\ref{sec:quad} we study functors from spaces to spaces concentrated in degrees $[i, 2i]$, and relate the attaching maps in their Goodwillie tower to the left action of $\partial_*(\Id)$.  In Section~\ref{sec:Hlayers} we recall from \cite{Behrens} the construction of homology operations $\bar{Q}^j$ which act on the stable homology of the derivatives of any functor from spaces to spaces, and their relationship to the Arone-Mahowald computation of the stable homology of the Goodwillie tower of the identity evaluated on spheres \cite{AroneMahowald}.  We also recall some homology calculations of \cite{Kuhn}.  The main theorems are proved in Section~\ref{sec:Hpsi}.

\begin{Ack}
The author would like to extend his heartfelt gratitude to Greg Arone and Nick Kuhn for generously sharing everything they knew about the subject of this paper, engaging in lively correspondence, and hosting a very productive visit to the University of Virginia.  The author also thanks Greg Arone, Bill Dwyer and Kathryn Lesh for informing the author of their alternative approach to the results of this paper.  The author benefited from discussions with Haynes Miller, who suggested he analyze the Rector cosimplicial model for the homotopy fiber, Michael Ching, who explained his work with Arone on operadic modules and chain rules in calculus, and Rosona Eldred, who alerted the author to the relationship between cosimplicial matching objects and cross effects.  Finally, the author is grateful to the referee, for his/her valuable suggestions, and to Mark Mahowald, for asserting to the author in 2003, ``The Goodwillie tower for $S^1$ is the Whitehead conjecture,''  sparking the author's desire to know ``why.''

This work was partially supported by the NSF and a grant from the Sloan foundation. 
\end{Ack}

\section{Generalized quadratic functors}\label{sec:quad}

For the purposes of this section, let $F$ be an analytic finitary homotopy functor
$$ F: \Top_* \rightarrow \Top_* $$
for which there exists an integer $i \ge 1$ so that the Goodwillie layers $D_k(F)$ are trivial unless $i \le k \le 2i$.
We regard such functors as ``generalized quadratic functors,'' as the operadic structure of their derivatives bears similarities to the quadratic case of $i = 1$.  In this section we analyze the relationship between the left action of $\partial_*(\Id)$ on $\partial_*(F)$ and the attaching maps between layers of the Goodwillie tower of $F$.  

There is only one potentially non-trivial component to the left action of $\partial_*(\Id)$ on $\partial_*(F)$: this is the map
\begin{equation}\label{eq:mu}
\mu: \partial_2(\Id) \wedge \partial_i(F) \wedge \partial_i(F) \rightarrow \partial_{2i}(F).
\end{equation}
We remind the reader that $\partial_2(\Id) \simeq S^{-1}$ (with trivial $\Sigma_2$ action).

There is a fiber sequence of functors
\begin{equation}\label{eq:phi}
F(X) \rightarrow P_{2i-1}(F)(X) \xrightarrow{\phi} B D_{2i}(F)(X).
\end{equation}
By \cite[Thm.~4.2]{AroneDwyerLesh}, the functor $P_{2i-1}(F)(X)$ admits a canonical infinite delooping
$$ P_{2i-1}(F)(X) \simeq \Omega^\infty \PP_{2i-1}(F)(X), $$
where $\PP_{2i-1}(F)$ is a spectrum valued functor.  The attaching map $\phi$ has an adjoint
$$ \td{\phi}: \Sigma^\infty \Omega^\infty \PP_{2i-1}(F)(X) \rightarrow \Sigma \DD_{2i}(F)(X). $$
Viewing $\td{\phi}$ as a natural transformation of functors $\Top_* \rightarrow \Sp$, there is an induced natural transformation on $2i$th layers of the Goodwillie towers of these functors: 
\begin{equation}\label{eq:tdphi2}
\td{\phi}_{2}: \DD_{2i}(\Sigma^\infty \Omega^\infty \PP_{2i-1}(F))(X) \rightarrow \Sigma\DD_{2i}(F)(X).
\end{equation}
The following lemma identifies the domain of $\td{\phi}_2$.

\begin{lem}\label{lem:tdphi2}
There is a natural equivalence
$$
\DD_{2i}(\Sigma^\infty \Omega^\infty \PP_{2i-1}(F))(X) \simeq \DD_i(F)(X)^{\wedge 2}_{h\Sigma_2}.  
$$
\end{lem}

\begin{proof}
The derivatives of the functor $\Sigma^\infty \Omega^\infty$ are well known to be given by
$$ \partial_i(\Sigma^\infty \Omega^\infty) =  S $$
with trivial $\Sigma_i$-action (see, for instance, \cite[Ex.~6.2]{KuhnChrom}).
We have (by the chain rule \cite{AroneChing})
\begin{align*}
\partial_{2i} (\Sigma^\infty \Omega^\infty \PP_{2i-1}(F)) & \simeq (\partial_*(\Sigma^\infty \Omega^\infty) \circ \partial_*(\PP_{2i-1}F))_{2i} \\
& \simeq {\Sigma_{2i}}_+ \underset{\Sigma_2 \wr \Sigma_i}{\wedge} \partial_{i} (F)^{\wedge 2}.
\end{align*}
We therefore have
\begin{align*}
\DD_{2i}(\Sigma^\infty\Omega^\infty \PP_{2i-1}(F))(X) 
& \simeq \partial_{2i} (\Sigma^\infty \Omega^\infty \PP_{2i-1}(F)) \wedge_{h\Sigma_{2i}} X^{\wedge 2i} \\ 
& \simeq [{\Sigma_{2i}}_+ \underset{\Sigma_2 \wr \Sigma_i}{\wedge} \partial_{i} (F)^{\wedge 2}] \wedge_{h\Sigma_{2i}} X^{\wedge 2i} \\
& \simeq \DD_i(F)(X)^{\wedge 2}_{h\Sigma_2}.  
\end{align*}
\end{proof}

Lemma~\ref{lem:tdphi2} allows us to regard $\td{\phi}_2$ as a map
$$ \td{\phi}_2:  \DD_i(F)(X)^{\wedge 2}_{h\Sigma_2}    \rightarrow \Sigma \DD_{2i}(F)(X). $$
Our main observation is the following.

\begin{thm}\label{thm:quad}
The map
$$ \Sigma^{-1}\td{\phi}_2: \Sigma^{-1} \DD_i(F)(X)^{\wedge 2}_{h\Sigma_2}    \rightarrow \DD_{2i}(F)(X) $$
is homotopic to the composite
\begin{align*} \Sigma^{-1} \DD_i(F)(X)^{\wedge 2}_{h\Sigma_2} 
& \simeq (\partial_2(\Id) \wedge \partial_i(F)^{\wedge 2} \wedge X^{\wedge 2i})_{h\Sigma_2 \wr \Sigma_i} \\ 
& \xrightarrow{\mu \wedge 1}  (\partial_{2i}(F) \wedge X^{\wedge 2i})_{h\Sigma_{2i}} \\
& \simeq \DD_{2i}(F)(X).
\end{align*}
\end{thm}

The proof of Theorem~\ref{thm:quad} will occupy the remainder of this section, and will require a series of  supporting lemmas.  At the heart of the argument is the following idea: given the attaching map $\phi$, compute the induced left action of $\partial_*(\Id)$ on $\partial_*(F)$.  This will result in a formula relating $\td{\phi}_2$ and $\mu$.

To compute the left action we use the machinery of Arone and Ching. 
For a functor $G: \Top_* \rightarrow \Top_*$, Arone and Ching \cite{AroneChing} show that $\partial_*(\Sigma^\infty G)$ is a left comodule over the commutative cooperad $\mr{Comm}_*$, and moreover show that $\partial_*(G)$ can be recovered from the cooperadic cobar construction
$$ \partial_*(G) \simeq  C(1_*, \mr{Comm}_*, \partial_*(\Sigma^\infty G)). $$
The cobar construction is Spanier-Whitehead dual to the bar construction 
\begin{equation}\label{eq:barcobar}
C(1_*, \mr{Comm}_*, \partial_*(\Sigma^\infty G)) \simeq B(1_*, \mr{Comm}_*, \partial^*(\Sigma^\infty G) )^\vee.
\end{equation}
Here, following \cite{AroneChing}, 
$$ \partial^*(\Sigma^\infty G) := \partial_*(\Sigma^\infty G)^\vee, $$ 
and must be interpreted as a symmetric sequence of pro-spectra for a general functor $G$.  Note that in in the right-hand side of (\ref{eq:barcobar}), we have abusively used $\mr{Comm}_*$ to also denote the commutative \emph{operad}, as it has the same underlying symmetric sequence as the commutative cooperad.
Ching's topological model for the bar construction \cite{Ching}
$$ B(1_*, \mr{Comm}_*, \partial^*(\Sigma^\infty G)) $$
carries a left coaction by the cooperad
$$ B(1_*, \mr{Comm}_*, 1_*) \simeq \partial^*(\Id). $$
The action of $\partial_*(\Id)$ induced on the dual recovers the left action of $\partial_*(\Id)$ on $\partial_*(G)$.    The proof of Theorem~\ref{thm:quad} will follow from an analysis of how this process plays out, when applied to our functor $F$.

The first step to the approach outlined above is to compute $\partial_*(\Sigma^\infty F)$.  The strategy is to use the fiber sequence (\ref{eq:phi}).  Note that as the functors $P_{2i-1}(F)$ and $BD_{2i}(F)$ factor through the category of spectra, the derivatives of $\Sigma^\infty$ of these functors are easily deduced from the derivatives of $\Sigma^\infty \Omega^\infty$ by applying the chain rule \cite{AroneChing}.
 
Since $F$ is analytic, for $X$ sufficiently highly connected there is a natural equivalence
\begin{equation}\label{eq:Rector}
 \Sigma^\infty F(X) \simeq h{\Tot} \Sigma^\infty (BD_{2i}(F)(X)^{\times \bullet} \times P_{2i-1}(F)(X))
\end{equation}
where 
$$ T^\bullet (X) := BD_{2i}(F)(X)^{\times \bullet} \times P_{2i-1}(F)(X) $$ 
is the Rector cosimplicial model \cite{Rector} for the homotopy fiber (\ref{eq:phi})
\begin{equation*}
\xymatrix@C-1.6em{
P_{2i-1}(F)(X) 
\ar@<1ex>[r]
\ar@<-1ex>[r] &
BD_{2i}(F)(X) \times P_{2i-1}(F)(X)
\ar[l]
\ar@<2ex>[r]
\ar@<0ex>[r]
\ar@<-2ex>[r] &
BD_{2i}(F)(X)^{\times 2} \times P_{2i-1}(F)(X) \cdots
\ar@<1ex>[l]
\ar@<-1ex>[l] 
}
\end{equation*}
and $h{\Tot}$ denotes homotopy totalization.

In preparation for our arguments, we briefly discuss some general properties of the homotopy $\Tot$-tower.  For a cosimplicial spectrum $Z^\bullet$, this tower takes the form:
$$ h\Tot^0 Z^\bullet \leftarrow h \Tot^1 Z^\bullet \leftarrow h \Tot^2 Z^\bullet \leftarrow \cdots . $$
Let $\mr{fib}^n Z^\bullet$ denote the homotopy fiber
$$  \mr{fib}^n Z^\bullet \rightarrow h\Tot^n Z^\bullet \rightarrow h\Tot^{n-1} Z^\bullet. $$
There are homotopy fiber sequences
\begin{equation*}
 \mr{fib}^n Z^\bullet \rightarrow Z^n \rightarrow \holim_{\substack{[n] \twoheadrightarrow [k] \\ k < n}} Z^k
\end{equation*}
Since a surjection $[n] \twoheadrightarrow [k]$ is uniquely determined by specifying the subset of arrows of 
$$ [n] = (0 \to 1 \to 2 \to \cdots \to n) $$
which go to identity arrows in $[k]$, the fiber $\mr{fib}^n Z^\bullet$ is computed as the total homotopy fiber of an $n$-cubical diagram
\begin{equation}\label{eq:matching}
\mr{fib}^n Z^\bullet \simeq h \mr{tfiber} \left\{ Z^{n-\abs{S}} \right\}_{S \subseteq \ul{n} }
\end{equation}
where the maps in the $n$-cubical diagram are given by codegeneracy maps of $Z^\bullet$.

Since $F$ was assumed to be analytic, there exists $\rho$, $q$ such that on sufficiently highly connected spaces $X$ the natural transformations
$$ F(X) \rightarrow P_k(F)(X) $$
are $(q-k(\rho-1)+(k+1)\mr{conn}(X))$-connected.  We will need the following connectivity estimate.

\begin{lem}\label{lem:connectivity}
On sufficiently highly connected spaces $X$, the map
$$ \Sigma^\infty F(X) \rightarrow h\Tot^n \Sigma^{\infty} T^\bullet (X) $$
is $(n+1)(q-(2i-1)(\rho-1))+1+2i(n+1)\mr{conn}(X)$-connected.
\end{lem}

\begin{proof}
Using the splitting
$$ \Sigma^\infty (Y \times Y') \simeq \Sigma^\infty Y \vee \Sigma^\infty Y' \vee \Sigma^\infty Y \wedge Y' $$
for $Y,Y' \in \Top_*$, one inductively computes from (\ref{eq:matching}) that 
$$ \mr{fib}^n \Sigma^\infty T^\bullet(X) \simeq \Sigma^\infty BD_{2i}(F)(X)^{\wedge n} \wedge P_{2i-1}(F)(X)_+. $$ 
For $\mr{conn}(X) \ge \rho$ the map
\begin{align*}
P_{2i}(F)(X) \rightarrow P_{2i-1}(F)(X)
\end{align*}
is $q-(2i-1)(\rho-1)+2i \cdot \mr{conn}(X)$-connected.  Therefore the fiber $D_{2i}(F)(X)$ is $q-(2i-1)(\rho-1)+2i \cdot \mr{conn}(X)-1$-connected, and the space $BD_{2i}(F)(X)$ is $q-(2i-1)(\rho-1)+2i \cdot \mr{conn}(X)$-connected.  Let $X$ be highly enough connected to make this number positive.  Then $\mr{fib}^n \Sigma^\infty T^\bullet$ is $n(q-(2i-1)(\rho-1))+2in \cdot \mr{conn}(X)$-connected.  We deduce that the map
$$ \Sigma^\infty F(X) \simeq h\Tot \Sigma^\infty T^\bullet (X) \rightarrow h\Tot^n \Sigma^\infty T^\bullet (X) $$
is $(n+1)(q-(2i-1)(\rho-1))+1+2i(n+1)\mr{conn}(X)$-connected.
\end{proof}

We are now able to identify $\partial_*(\Sigma^\infty F)$ for $* \le 2i$.

\begin{lem}\label{lem}\label{lem:dSigmainfty}
There are equivalences 
\begin{align*}
\partial_k(\Sigma^\infty F) & \simeq \partial_k(F), \quad \text{for $i \le k < 2i$,} \\
\partial_{2i}(\Sigma^\infty F) & \simeq \mr{fiber}\left(
{\Sigma_{2i}}_+ \underset{\Sigma_2 \wr \Sigma_i}{\wedge} \partial_{i} (F)^{\wedge 2} \xrightarrow{\partial_{2i}(\td{\phi}_2)} \Sigma \partial_{2i}(F) \right).
\end{align*}
\end{lem}

\begin{proof}
Recall from the proof of Lemma~\ref{lem:tdphi2} that we used the chain rule to deduce
$$ \partial_{2i} (\Sigma^\infty P_{2i-1}(F)) \simeq {\Sigma_{2i}}_+ \underset{\Sigma_2 \wr \Sigma_i}{\wedge} \partial_{i} (F)^{\wedge 2}. $$
The same argument shows that: 
\begin{align*}
\partial_k (\Sigma^\infty P_{2i-1}(F)) & \simeq \partial_k(F), \quad \text{for $i \le k < 2i$}, \\
\partial_{2i}(\Sigma^\infty BD_{2i}(F)) & \simeq \Sigma \partial_{2i}(F).
\end{align*}
By Lemma~\ref{lem:connectivity}, the functors $\Sigma^\infty F$ and $h \Tot^0 \Sigma^\infty T^\bullet$ agree to order $2i-1$ and the functors $\Sigma^\infty F$ and $h \Tot^1 \Sigma^\infty T^\bullet$ agree to order $2i$.  It follows \cite{Goodwillie} that
\begin{align*}
\partial_k(\Sigma^\infty F) & \simeq h \Tot^0 \partial_k(\Sigma^\infty T^\bullet), \quad k < 2i, \\
\partial_{2i}(\Sigma^\infty F) & \simeq h \Tot^1 \partial_{2i}(\Sigma^\infty T^\bullet).
\end{align*}
This immediately implies the first equivalence of the lemma.

To prove the second equivalence, we must compute $h \Tot^1 \partial_{2i}(\Sigma^\infty T^\bullet)$.
The $\partial_{2i}$ computations above imply that 
\begin{equation}\label{eq:cosimplicial}
 \partial_{2i}(\Sigma^\infty T^s) \simeq \underbrace{\Sigma \partial_{2i}(F) \vee \cdots \vee \Sigma \partial_{2i}(F)}_{s} \: \vee \: {\Sigma_{2i}}_+ \underset{\Sigma_2 \wr \Sigma_i}{\wedge} \partial_{i} (F)^{\wedge 2}.
\end{equation}

We claim that under the equivalences (\ref{eq:cosimplicial}), the last coface map in the cosimplicial $\Sigma_{2i}$-spectrum $\partial_{2i}(\Sigma^\infty T^\bullet)$  from level $0$ to level $1$ is given by
$$ d^1 = \partial_{2i}(\td{\phi}_2) \times 1, $$
and the codegeneracy map from level $1$ to level $0$ is the map which collapses out the wedge summand $\Sigma \partial_{2i}(F)$.  The second equivalence of the lemma follows immediately from this claim.

To establish the claim concerning the cosimplicial structure maps above, 
observe that the $d^{1}$ map from level $0$ to level $1$ in the cosimplicial functor $\Sigma^\infty T^\bullet(X)$ is the composite
\begin{align*}
\delta: \Sigma^{\infty} P_{2i-1}(F)(X) 
& \xrightarrow{\Sigma^\infty \Delta} \Sigma^{\infty} \left( P_{2i-1}(F)(X) \times P_{2i-1}(F)(X) \right) \\  
& \xrightarrow{\Sigma^\infty \phi \times 1} \Sigma^{\infty} \left( BD_{2i}(F)(X) \times P_{2i-1}(F)(X) \right).
\end{align*}
The induced map $\partial_{2i}(\delta)$ is determined by the composites with the projections onto the wedge summands of 
$$ \partial_{2i}(\Sigma^{\infty} ( BD_{2i}(F) \times P_{2i-1}(F)) \simeq 
\Sigma \partial_{2i}(F) \vee {\Sigma_{2i}}_+ \underset{\Sigma_2 \wr \Sigma_i}{\wedge} \partial_i(F)^{\wedge 2}_{h\Sigma_2}.
$$
Composing $\delta$ with the projection onto the second factor gives the identity, and this implies that the second component of $\partial_{2i}(\delta)$ is the identity.  Composing $\delta$ with the projection onto the first factor is the natural transformation
$$ \Sigma^\infty \phi: \Sigma^\infty P_{2i-1}(F)(X) \rightarrow \Sigma^\infty BD_{2i}(F)(X). $$
Using the fact that the adjoint $\td{\phi}$ is the composite
$$ \Sigma^\infty \Omega^{\infty} \PP_{2i-1}(F)(X) \xrightarrow{\Sigma^\infty \phi} \Sigma^\infty \Omega^\infty \Sigma \DD_{2i}(F)(X) \xrightarrow{\epsilon} \Sigma \DD_{2i}(F)(X), $$
together with the fact that $\epsilon$ is a $\partial_{2i}$-equivalence, we deduce that the first component of $\partial_{2i}(\delta)$ is $\partial_{2i}(\td{\phi}_2)$, as desired.
The claim concerning the codegeneracy of $\partial_{2i}(\Sigma^\infty T^\bullet)$ follows immediately from the fact that the codegeneracy from level $1$ to level $0$ of the cosimplicial functor $T^\bullet(X)$ projects away the first component.
\end{proof}

The last equivalence of Lemma~\ref{lem:dSigmainfty} gives a fiber sequence of $\Sigma_{2i}$-spectra:
\begin{equation}\label{eq:d2ifiber}
\partial_{2i}(F) \xrightarrow{\eta} \partial_{2i}(\Sigma^\infty F) \xrightarrow{\xi} {\Sigma_{2i}}_+ \underset{\Sigma_2 \wr \Sigma_i}{\wedge} \partial_{i} (F)^{\wedge 2} \xrightarrow{\partial_{2i}(\td{\phi}_2)} \Sigma \partial_{2i}(F).
\end{equation}

Our next task is to understand the
left coaction of $\mr{Comm}_*$ on $\partial_*(\Sigma^\infty F)$ in low degrees in terms of the attaching map $\phi$. 

\begin{lem}
Under the equivalence $\partial_i(F) \simeq \partial_i(\Sigma^\infty F)$, the map $\xi$ of (\ref{eq:d2ifiber}) agrees with the left $\mr{Comm}_*$-comodule structure map
$$ \partial_{2i}(\Sigma^\infty F) \rightarrow {\Sigma_{2i}}_+ \underset{\Sigma_2 \wr \Sigma_i}{\wedge} \mr{Comm}_2 \wedge \partial_{i} (\Sigma^\infty F)^{\wedge 2}. $$
\end{lem}

\begin{proof}
 The left coaction of $\mr{Comm}_*$ on 
$$ \partial_*(\Sigma^\infty P_{2i-1}(F)) = \partial_*(\Sigma^\infty \Omega^\infty \PP_{2i-1}(F)) $$
is easily deduced from the chain rule \cite{AroneChing}, together with the fact that under the equivalence
$$ \partial_*(\Sigma^\infty \Omega^\infty) \simeq \mr{Comm}_*, $$
the left coaction of $\mr{Comm}_*$ on $\partial_*(\Sigma^\infty \Omega^\infty)$ is given by the left coaction of $\mr{Comm}_*$ on itself.  In particular, the coaction map corresponding to the partition $2i = i+i$ is given by the composite (of equivalences)
$$
\partial_{2i}(\Sigma^\infty P_{2i-1}(F)) \xrightarrow{\simeq} {\Sigma_{2i}}_+ \underset{\Sigma_2 \wr \Sigma_i}{\wedge} \partial_{i} (F)^{\wedge 2}  \xrightarrow{\simeq} {\Sigma_{2i}}_+ \underset{\Sigma_2 \wr \Sigma_i}{\wedge} \partial_{i} (\Sigma^\infty P_{2i-1}(F))^{\wedge 2}.
$$

The natural transformation of functors
$$ F \rightarrow P_{2i-1}(F) $$
induces a map of left $\mr{Comm}_*$-comodules
$$ \partial_*(\Sigma^\infty F) \rightarrow \partial_*(\Sigma^\infty P_{2i-1}(F)). $$
In particular, there is a commutative diagram
$$
\xymatrix{
\partial_{2i}(\Sigma^\infty F) \ar[r]^-\xi \ar[d] &
{\Sigma_{2i}}_+ \underset{\Sigma_2 \wr \Sigma_i}{\wedge} \partial_{i} (F)^{\wedge 2} \ar[d]^{=}
\\
{\Sigma_{2i}}_+ \underset{\Sigma_2 \wr \Sigma_i}{\wedge} \partial_{i} (\Sigma^\infty F)^{\wedge 2} \ar[r]_{\simeq} &
{\Sigma_{2i}}_+ \underset{\Sigma_2 \wr \Sigma_i}{\wedge} \partial_{i} (F)^{\wedge 2}
}
$$
where the vertical arrows are $\mr{Comm}_*$-comodule structure maps.  We conclude that the map $\xi$ in (\ref{eq:d2ifiber}) encodes the primary $\mr{Comm}_*$-comodule structure map, as desired.
\end{proof}

\begin{proof}[Proof of Theorem~\ref{thm:quad}]
By \cite{AroneChing}, we have
\begin{equation}\label{eq:AroneChing}
\partial_*(F) \simeq C(1_*, \mr{Comm}_*, \partial_*(\Sigma^\infty F)).
\end{equation}
In particular, we have
\begin{align*}
\partial_{2i}(F) 
& \simeq C(1_*, \mr{Comm}_*, \partial_*(\Sigma^\infty F))_{2i} \\
& \simeq \mr{fiber}
\left(
\partial_{2i}(\Sigma^\infty F) \xrightarrow{\xi}
{\Sigma_{2i}}_+ \underset{\Sigma_2 \wr \Sigma_i}{\wedge} \partial_{i} (F)^{\wedge 2}
\right).
\end{align*}
This equivalence was already recorded in the fiber sequence (\ref{eq:d2ifiber}), but now it implicitly records more structure, as (\ref{eq:AroneChing}) is an equivalence of left $\partial_*(\Id)$-modules.  Indeed, we now compute from (\ref{eq:AroneChing}) the $\partial_*(\Id)$-module structure of $\partial_*(F)$ in terms of the attaching map $\phi$.

To accomplish this, we work with dual derivatives, and then dualize.  We have
$$ \partial^*F = B(1_*, \mr{Comm}_*, \partial^*(\Sigma^\infty F)). $$
Using the Ching model for the bar construction \cite{Ching}, we have a pushout
$$
\xymatrix{
\partial I \wedge {\Sigma_{2i}}_+ 
\underset{\Sigma_2 \wr \Sigma_i}{\wedge} \partial^i(F) \ar@{^{(}->}[r] \ar[d]_{\xi^\vee} &
I \wedge {\Sigma_{2i}}_+ 
\underset{\Sigma_2 \wr \Sigma_i}{\wedge} \partial^i(F) \ar[d] 
\\
\partial I \wedge \partial^{2i}(\Sigma^\infty F) \ar[r] &
 B(1_*, \mr{Comm}_*, \partial^*(\Sigma^\infty F))_{2i}
}
$$
and the $\partial^*(\Id)$-comodule structure map is explicitly given by the map of pushouts
$$
\xymatrix@C-5em{
\partial I \wedge {\Sigma_{2i}}_+ 
\underset{\Sigma_2 \wr \Sigma_i}{\wedge} \partial^i(F) \ar@{^{(}->}[rr] \ar[dd]_{\xi^\vee} \ar[dr]^= &&
I \wedge {\Sigma_{2i}}_+ 
\underset{\Sigma_2 \wr \Sigma_i}{\wedge} \partial^i(F) \ar[dd]|-{\phantom{X}} \ar[dr]^= 
\\ 
& \partial I \wedge {\Sigma_{2i}}_+ 
\underset{\Sigma_2 \wr \Sigma_i}{\wedge} \partial^i(F) \ar@{^{(}->}[rr] \ar[dd] && 
I \wedge {\Sigma_{2i}}_+ 
\underset{\Sigma_2 \wr \Sigma_i}{\wedge} \partial^i(F) \ar[dd]
\\
\partial I \wedge \partial^{2i}(\Sigma^\infty F) \ar[rr]|{\quad} \ar[dr] &&
\partial^{2i}(F) \ar[dr] 
\\
& \ast \ar[rr] && {\Sigma_{2i}}_+ \underset{\Sigma_2 \wr \Sigma_i}{\wedge} \partial^{2}(\Id) \wedge \partial^i(F)^{\wedge 2}
}
$$
In particular, we deduce that the coaction map
$$ \mu^\vee: \partial^{2i}(F) \rightarrow {\Sigma_{2i}}_+ \underset{\Sigma_2 \wr \Sigma_i}{\wedge} \partial^{2}(\Id) \wedge \partial^i(F)^{\wedge 2}  $$
is precisely the connecting morphism $(\Sigma^{-1} \partial_{2i}(\td{\phi}_2))^\vee$ in the cofiber sequence dual to the fiber sequence (\ref{eq:d2ifiber}):
$$
{\Sigma_{2i}}_+ \underset{\Sigma_2 \wr \Sigma_i}{\wedge} \partial^i(F)^{\wedge 2} \xrightarrow{\xi^\vee} \partial^{2i}(\Sigma^\infty F) \xrightarrow{\eta^\vee} \partial^{2i}(F) \xrightarrow{(\Sigma^{-1} \partial_{2i}(\td{\phi}_2))^\vee }
\Sigma \left( {\Sigma_{2i}}_+ \underset{\Sigma_2 \wr \Sigma_i}{\wedge} \partial^i(F)^{\wedge 2} \right).
$$
Dualizing, we deduce that 
$$ \mu = \Sigma^{-1}\partial_{2i}(\td{\phi}_2) $$
and the theorem follows.
\end{proof}

\section{Homology of the layers}\label{sec:Hlayers}

In this section we briefly recall some facts about the homology of the layers of the Goodwillie tower of the identity evaluated on spheres.  This computation is due to Arone and Mahowald \cite{AroneMahowald}, but we will need to take advantage of the interpretation presented in \cite{Behrens}.

Let $F: \Top_* \rightarrow \Top_*$ be a reduced finitary homotopy functor.
In \cite{Behrens}, the author introduced operations
$$ \bar{Q}^j: H_*(\DD_i(F)(X)) \rightarrow H_{*+j-1}(\DD_{2i}(F)(X)). $$
These operations were defined as follows: the left action of $\partial_*(\Id)$ on $\partial_*(F)$ yields a map
$$ \mu : \Sigma^{-1} \partial_i(F)^{\wedge 2} \simeq \partial_2(\Id) \wedge \partial_i(F)^{\wedge 2} \rightarrow \partial_{2i}(F). $$
This map induces a map
$$ \mu': \Sigma^{-1} \DD_{i}(F)(X)^{\wedge 2}_{h\Sigma_2} \rightarrow \DD_{2i}(F)(X). $$
The operations are given by 
\begin{equation}\label{eq:Qbar}
\bar{Q}^j(x) :=  \mu'_* \sigma^{-1} Q^j x
\end{equation}
for $x \in H_*(\DD_i(F)(X))$.

In \cite{Behrens}, the Arone-Mahowald computation is interpreted in terms of these operations, and it is shown that
$$ H_*(\DD_{2^k}(S^n)) = \FF_2 \left\{ \bar{Q}^{i_1} \cdots \bar{Q}^{i_k} \iota_n \: : \: 
i_s \ge 2i_{s+1}+1, \: i_k \ge n   
\right\}. $$
Recall that $H_* (S^1)^{\wedge 2^k}_{h\Sigma^{\wr k}_{2}}$ contains a direct summand
$$ \td{\mc{R}}_1(k) = \FF_2\{ Q^{i_1} \wr \cdots \wr Q^{i_k} \iota_1 \: : \: i_s \ge i_{s+1}+ \cdots + i_k+1\}. $$
In \cite{Kuhn}, certain idempotents $e_{k}$ are constructed to act on $\td{\mc{R}}_1(k)$ (in \cite{Kuhn}, these idempotents are denoted $D_{k-1}$, but we use the notation $e_k$ in this paper so as to not create confusion with the notation used for the layers of the Goodwillie tower).
These idempotents split off the summand $H_*(\Sigma L(k))$.  Kuhn shows that
$$ H_*(\Sigma L(k)) = \FF_2 \left\{ e_k(Q^{i_1} \wr \cdots \wr {Q}^{i_k} \iota_1) \: : \: 
i_s \ge 2i_{s+1}+1, \: i_k \ge 1    
\right\}. $$

\begin{lem}\label{lem:basis}
Under the equivalence $\Sigma L(k) \simeq \Sigma^k \mb{D}_{2^k}(S^1)$ of (\ref{eq:AroneDwyer}), we have a bijection between the two bases
$$ e_k(Q^{i_1} \wr \cdots \wr {Q}^{i_k} \iota_1) \leftrightarrow \sigma^k\bar{Q}^{i_1} \cdots \bar{Q}^{i_k} \iota_1. $$
\end{lem}

\begin{proof}
In Section~1.4 of \cite{Behrens} an algebra $\bar{\mc{R}}_n$ of operations $\bar{Q}^j$ is defined, with relations
\begin{align*}
(1) \quad & \bar{Q}^r \bar{Q}^s = \sum_t \left[ \binom{s-r+t}{s-t} + \binom{s-r+t}{2t-r} \right] \bar{Q}^{r+s-t} \bar{Q}^{t}, \\
(2) \quad & \bar{Q}^{j_1} \cdots \bar{Q}^{j_k} = 0, \: \text{if $j_1 < j_2+\cdots +j_k +n$}.
\end{align*}
Here, and throughout this section, mod $2$ binomial coefficients $\binom{a}{b} \in \FF_2$ are defined for all $a, b \in \ZZ$ by
$$ \binom{a}{b} = \text{coefficient of $t^b$ in $(1+t)^a$.} $$
Let $\bar{\mc{R}}_n(k)$ be the summand additively generated by length $k$ sequences of operations.  
It is shown in \cite{Behrens} that $H_* \Sigma^k \DD_{2^k}(S^1)$ is precisely the quotient of $\td{\mc{R}}_1(k)$ by relation (1) above, and therefore
$$ H_* \DD_{2^k}(S^1) = \bar{\mc{R}}_1(k)\{ \iota_1 \}. $$  

Kuhn's idempotents $e_k$ are defined in \cite{Kuhn} as certain iterates of idempotents
$$ T_s: \td{\mc{R}}_1 \rightarrow \td{\mc{R}}_1, \quad 1 \le s \le k-1 $$
where 
\begin{multline*}
T_s(Q^{i_1} \wr \cdots \wr Q^{i_k}) = 
\\
\sum_t \left[ \binom{i_{s+1}-i_s+t}{i_{s+1}-t} + \binom{i_{s+1}-i_s+t}{2t-i_s} \right] {Q}^{i_1} \wr \cdots \wr \bar{Q}^{i_s+i_{s+1}-t} \wr {Q}^{t} \wr \cdots \wr Q^{i_k}.
\end{multline*}
Let
\begin{equation}\label{eq:nuk}
 \nu_k: \td{\mc{R}}_1(k) \rightarrow \bar{\mc{R}}_1(k)
\end{equation}
be the canonical surjection.  Clearly
$\nu_k T_s = \nu_k$, and therefore $\nu_k e_{k} = \nu_k$.  In particular
$$ \nu_k  e_{k} Q^{i_1} \wr \cdots \wr Q^{i_k} = \bar{Q}^{i_1} \cdots \bar{Q}^{i_k}. $$
\end{proof}

We remark that the above lemma, and the homomorphisms used in its proof, allow us to easily describe the effect on homology
\begin{gather*}
(\iota_k)_*: H_* \Sigma^k \DD_{2^k} (S^1) \rightarrow H_*(S^1)^{\wedge 2^k}_{h\Sigma_2^{\wr k}}, \\
(p_k)_*:  H_*(S^1)^{\wedge 2^k}_{h\Sigma_2^{\wr k}} \rightarrow H_* \Sigma^k \DD_{2^k} (S^1)
\end{gather*}
 of the splitting maps (\ref{eq:summand}) in terms of the $\bar{Q}$-basis.  Namely, we have
\begin{align}
(\iota_k)_* \sigma^k\bar{Q}^{i_1} \cdots \bar{Q}^{i_k} \iota_1 & = e_k(Q^{i_1} \wr \cdots \wr Q^{i_k} \iota_1) \label{eq:Hi_k} \\
(p_k)_*Q^{i_1} \wr \cdots \wr Q^{i_k} \iota_1 & = \sigma^k \bar{Q}^{i_1} \cdots \bar{Q}^{i_k} \iota_1. \label{eq:Hp_k}
\end{align}

Since (\ref{eq:AroneDwyer}) and (\ref{eq:summand}) give the spectrum
$\Sigma^k \DD_{2^k}(S^1)$ 
as a summand of the suspension spectrum $\Sigma^\infty (S^1)^{\wedge 2^k}_{h\Sigma_{2}^{\wr k}}$, we can explicitly describe the homology of its zeroth space \cite{CohenLadaMay}:
$$ H_*(B^k D_{2^k}(S^1)) = \mc{F} \left (\FF_2\left\{ \sigma^k \bar{Q}^{i_1} \cdots \bar{Q}^{i_k} \iota_1 \: : \: 
i_s \ge 2i_{s+1}+1, \: i_k \ge 1   
\right\} \right). $$
Here, $\mc{F}$ is the functor
$$ \mc{F}: \text{graded $\FF_2$-vector spaces} \rightarrow \text{allowable $\mc{R}$-algebras}  $$
which associates to a graded $\FF_2$-vector space $V$ the free allowable algebra over the Dyer-Lashof algebra.

We endow $H_*(B^kD_{2^k}(S^1))$ with a (decreasing) weight filtration by declaring that
\begin{align*}
w(x) & = 2^k \quad \text{for $x \in H_*\Sigma^k \DD_{2^k}(S^1)$}, \\
w(Q^i x) & = 2 \cdot w(x), \\
w(x \ast y) & = w(x)+w(y).
\end{align*}  

This weight filtration is related to Goodwillie calculus in the following manner.  As indicated in the proof of Lemma~\ref{lem:tdphi2}, the functor $\Sigma^\infty \Omega^\infty$
has derivatives 
$$ \partial_i(\Sigma^\infty \Omega^\infty) =  S $$
with trivial $\Sigma_i$-action. For connected spectra $E$, the Goodwillie tower $P_i(\Sigma^\infty \Omega^\infty)(E)$ converges, giving a spectral sequence
\begin{equation}\label{eq:SS}
 E_1^{i,*} = H_*(E^{\wedge i}_{h\Sigma_i}) \Rightarrow H_*(\Omega^\infty E).
\end{equation}
In \cite[Ex.~6.1]{KuhnChrom}, it is explained that the Goodwillie tower for $\Sigma^\infty \Omega^\infty$ splits when evaluated on connected suspension spectra.  In these cases the spectral sequence (\ref{eq:SS}) degenerates.  By naturality, this also holds for summands of connected suspension spectra.  The weight filtration is simply an appropriate scaling of the filtration in this spectral sequence.

The induced morphisms
$$ 
H_* B^k D_{2^k}(S^1) 
\begin{array}{c}
(d_k)_* \\
\leftrightarrows \\
(\delta_k)_*
\end{array}
H_* B^{k+1} D_{2^{k+1}}(S^1) 
$$
were computed in \cite{Kuhn}: we end this section be recalling these explicit descriptions.

Suppose that 
$$ Q^{j_1} \cdots Q^{j_\ell}\sigma^{k+1}\bar{Q}^{i_1} \cdots \bar{Q}^{i_{k+1}}\iota_1 $$ 
is an algebra generator of $H_* B^{k+1} D_{2^{k+1}}(S^1)$.
Writing 
$$  e_{k+1} Q^{i_1} \wr \cdots \wr Q^{i_{k+1}} = \sum Q^{i'_1} \wr \cdots \wr Q^{i'_{k+1}}, $$
we have
$$ (d_k)_* Q^{j_1} \cdots Q^{j_\ell}\sigma^{k+1}\bar{Q}^{i_1} \cdots \bar{Q}^{i_{k+1}}\iota_1 = 
\sum Q^{j_1} \cdots Q^{j_\ell} Q^{i'_1} \sigma^k \bar{Q}^{i'_2}  \cdots  \bar{Q}^{i'_{k+1}}. $$
Furthermore, as $d_k$ is an infinite loop map, $(d_k)_*$ is a map of algebras.  We see that $(d_k)_*$ preserves the weight filtration.  In fact, $(d_k)_*$ is isomorphic to its own associated graded (with respect to the monomial basis).

Suppose that 
$$ Q^{j_1} \cdots Q^{j_\ell}\sigma^{k}\bar{Q}^{i_1} \cdots \bar{Q}^{i_{k}}\iota_1 $$ 
is an algebra generator of $H_* B^{k} D_{2^{k}}(S^1)$.  Then we have 
$$ (\delta_k)_* Q^{j_1} \cdots Q^{j_\ell}\sigma^{k}\bar{Q}^{i_1} \cdots \bar{Q}^{i_{k}}\iota_1 
= \sum_s Q^{j_1} \cdots \bar{Q}^{j_s} \cdots Q^{j_\ell}\sigma^{k}\bar{Q}^{i_1} \cdots \bar{Q}^{i_{k}}\iota_1.
$$
Here, we move the $\bar{Q}^i$ past the $Q^j$'s using the mixed Adem relation
$$ 
\bar{Q}^r {Q}^s = \sum_t \left[ \binom{s-r+t}{s-t} + \binom{s-r+t}{2t-r} \right] {Q}^{r+s-t} \bar{Q}^{t}.
$$
In particular, $(\delta_k)_*$ preserves the weight filtration on algebra generators.
While the map $(\delta_k)_*$ is \emph{not} a map of algebras, Kuhn shows that its associated graded with respect to the weight filtration \emph{is} a map of algebras \cite[Prop.~2.7]{Kuhn}.

\section{Homological behavior of $\psi_k$}\label{sec:Hpsi}

In this section we will prove Theorem~\ref{thm:psidelta}, and then explain how it implies Theorem~\ref{thm:main}.

The map $\delta_k$ is given by the composite
\begin{equation}\label{eq:deltafactor}
\Omega^\infty \Sigma^k \DD_{2^k}(S^1)  \xrightarrow{JH} \Omega^\infty \left( \Sigma^k \DD_{2^k}(S^1) \right)^{\wedge 2}_{h\Sigma_2} \xrightarrow{\Omega^\infty \alpha_k} \Omega^{\infty}\Sigma^{k+1} \DD_{2^{k+1}}(S^1)
\end{equation}
where $\alpha_k$ is the composite (see (\ref{eq:AroneDwyer}) and (\ref{eq:summand}))
$$  
\left( \Sigma^k \DD_{2^k}(S^1) \right)^{\wedge 2}_{h\Sigma_2} \xrightarrow{(\iota_k)^{\wedge 2}_{h\Sigma_2}} \Sigma^\infty (S^1)^{\wedge 2^{k+1}}_{h\Sigma_2^{\wr k+1}} \xrightarrow{p_{k+1}} \Sigma^{k+1} \DD_{2^{k+1}}(S^1) $$
and the James-Hopf map $JH$ is defined by the splitting of $\Sigma^\infty \Omega^\infty \Sigma^k \DD_{2^k}(S^1)$ induced by the retract of the Goodwillie towers
$$
\xymatrix@C-2em{ P_{i}(\Sigma^\infty \Omega^\infty)(\Sigma^k \DD_{2^k}(S^1)) \ar[dr]_{(\iota_k)_*} \ar[rr]^{\mital{id}}&&
P_{i}(\Sigma^\infty \Omega^\infty)(\Sigma^k \DD_{2^k}(S^1)) \\
& P_i(\Sigma^\infty \Omega^\infty)( \Sigma^\infty (S^1)^{\wedge 2^k}_{h\Sigma_2^{\wr k}}) \ar[ur]_{(p_k)_*} }
$$

Consider the natural transformation of functors from vector spaces to spectra given by the the adjoint of $\psi_k$:
$$ \td{\psi}_k: \Sigma^\infty \Omega^{\infty} \Sigma^k \DD_{2^k}(S^V) \rightarrow \Sigma^{k+1} \DD_{2^{k+1}}(S^V). $$
On the level of the $2^{k+1}$st layers of the corresponding Weiss towers, $\td{\psi}_k$ induces a map
$$ [\td{\psi}_k]_2 : (\Sigma^{k} \DD_{2^k}(S^V))^{\wedge 2}_{h\Sigma_2} \simeq \DD^W_{2^{k+1}}(\Sigma^\infty \Omega^{\infty} \Sigma^k \DD_{2^k} \circ \chi)(V) \rightarrow \Sigma^{k+1} \DD_{2^{k+1}}(S^V). $$
The proof of Theorem~\ref{thm:psidelta} will rest on the following two lemmas.

\begin{lem}\label{lem:lemma1}
The natural transformation $\psi_k$, when evaluated on $S^1$, admits a factorization
$$ \Omega^\infty \Sigma^k \DD_{2^k}(S^1)  \xrightarrow{JH} \Omega^\infty \left( \Sigma^k \DD_{2^k}(S^1) \right)^{\wedge 2}_{h\Sigma_2} \xrightarrow{\Omega^\infty [\td{\psi}_k]_2} \Omega^{\infty}\Sigma^{k+1} \DD_{2^{k+1}}(S^1).
$$
\end{lem}

\begin{proof}[Proof of Lemma~\ref{lem:lemma1}]
Since the functor $\Sigma^{k+1}\DD_{2^{k+1}}(S^V)$ is of degree $2^{k+1}$ in $V$, the adjoint $\td{\psi}_k$ factors as
$$
\Sigma^\infty \Omega^\infty \Sigma^{k}\DD_{2^k}(S^V) \rightarrow P^W_{2^{k+1}}(\Sigma^\infty \Omega^\infty \Sigma^k \DD_{2^k} \circ \chi)(V) \xrightarrow{\tau_k} \Sigma^{k+1} \DD_{2^{k+1}}(S^V)
$$
Specializing to the case of $V = \RR$, and using the splitting
\begin{align*}
P^W_{2^{k+1}}(\Sigma^\infty \Omega^\infty \Sigma^k \DD_{2^k} \circ \chi)(\RR)
& \simeq P_2(\Sigma^\infty \Omega^\infty)(\Sigma^k \DD_{2^k}(S^1)) \\
& \simeq \Sigma^k \DD_{2^k}(S^1) \vee \left( \Sigma^k \DD_{2^k}(S^1) \right)^{\wedge 2}_{h\Sigma_2} 
\end{align*}
we see that in this case $\tau_k$ may be decomposed as
$$ \tau_k = [\td{\psi}_k]_1 \vee [\td{\psi}_k]_2. $$
Using \cite[Cor.~5.4]{Nishida}, we see that 
$$ [\Sigma^k \DD_{2^k}(S^1), \Sigma^{k+1} \DD_{2^{k+1}}(S^1)] \cong [L(k), L(k+1)] = 0. $$
Therefore $[\td{\psi}_k]_1 \simeq \ast$, and the lemma follows.
\end{proof}

\begin{lem}\label{lem:lemma2}
The induced maps 
$$ (\alpha_k)_*, ([\td{\psi}_k]_2)_*: H_* \left( \Sigma^k \DD_{2^k}(S^1) \right)^{\wedge 2}_{h\Sigma_2} \rightarrow H_* \Sigma^{k+1} \DD_{2^{k+1}}(S^1). $$
are equal.
\end{lem}

\begin{proof}[Proof of Lemma~\ref{lem:lemma2}]
The map 
$$ (\alpha_k)_*: H_* \left( \Sigma^k \DD_{2^k}(S^1) \right)^{\wedge 2}_{h\Sigma_2} 
\rightarrow H_* \Sigma^{k+1} \DD_{2^{k+1}}(S^1) $$
can be computed using $(\ref{eq:Hi_k})$, $(\ref{eq:Hp_k})$, and the relation $\nu_k T_s = \nu_k$ established in the proof of Lemma~\ref{lem:basis}.  One finds that $(\alpha_k)_*$ 
is given by
$$ (\alpha_k)_* Q^j \sigma^k \bar{Q}^{i_1} \cdots \bar{Q}^{i_k} \iota_1 = \sigma^{k+1} \bar{Q}^{j} \bar{Q}^{i_1} \cdots \bar{Q}^{i_k} \iota_1. $$
We just need to show that the same holds for $[\td{\psi}_k]_2$.  

Since $\phi_k = \Omega^k \psi_k$, the evaluation maps $\Sigma^k \Omega^k \rightarrow \Id$ allow one to fit the adjoints $\td{\phi}_k$, $\td{\psi}_k$ of these natural transformations into the following commutative diagram
$$
\xymatrix{\Sigma^k \Sigma^\infty \Omega^\infty \DD_{2^k}(S^V)  \ar[r]^{\Sigma^k \td{\phi}_k} \ar[d]_{E^k} & \Sigma^{k+1} \DD_{2^{k+1}}(S^V) \ar[d]^= \\ \Sigma^\infty \Omega^\infty \Sigma^k \DD_{2^k}(S^V)  \ar[r]_{\td{\psi}_k} & \Sigma^{k+1} \DD_{2^{k+1}}(S^V) }
$$
On the level of $2^{k+1}$st Weiss layers, evaluated on $V = \RR$, we get a diagram
$$
\xymatrix{\Sigma^k \left( \DD_{2^k}(S^1) \right)^{\wedge 2}_{h\Sigma_2} \ar[r]^{\Sigma^k [\td{\phi}_k]_2} \ar[d]_{E^k} & \Sigma^{k+1} \DD_{2^{k+1}}(S^1) \ar[d]^= \\ \left( \Sigma^k \DD_{2^k}(S^1) \right)^{\wedge 2}_{h\Sigma_2} \ar[r]_{[\td{\psi}]_2} & \Sigma^{k+1} \DD_{2^{k+1}}(S^1) }
$$
The map 
$$ (E^k)_*: H_* \Sigma^k \left( \DD_{2^k}(S^1) \right)^{\wedge 2}_{h\Sigma_2} \rightarrow H_* \left( \Sigma^k \DD_{2^k}(S^1) \right)^{\wedge 2}_{h\Sigma_2} $$
is surjective, since $Q^i$-operations commute with $(E^k)_*$ (see, for example, \cite[Lem.~II.5.6]{Hinfty}).  Therefore, it suffices to compute 
$$ ([\td{\phi}_k]_2)_*: H_* \left( \DD_{2^k}(S^1) \right)^{\wedge 2}_{h\Sigma_2} \rightarrow H_* \Sigma \DD_{2^{k+1}}(S^1). $$
We compute this map using the technology of Section~\ref{sec:quad}.

Let $P_{2^k, 2^{k+1}}(X)$ be the generalized quadratic functor defined by the fiber sequence
$$  P_{2^k, 2^{k+1}}(X) \rightarrow P_{2^{k+1}}(X) \rightarrow P_{2^k-1}(X). $$
Then, as explained in Section~\ref{sec:quad}, there is a fiber sequence
$$ P_{2^k, 2^{k+1}}(X) \rightarrow \Omega^\infty \PP_{2^k, 2^{k+1}-1}(X) \xrightarrow{\phi_k} \Omega^\infty \Sigma \DD_{2^{k+1}}(X). $$
Here we have purposefully abused notation, as this new attaching map $\phi_k$ agrees with the old $\phi_k$ when $X$ is a sphere.  Associated to the adjoint of $\phi_k$ is a transformation
$$ [\td{\phi}_k]_2: \DD_{2^k}(X)^{\wedge 2}_{h\Sigma_2} \rightarrow \Sigma \DD_{2^{k+1}}(X) $$
which reduces to the previously defined $[\td{\phi}_k]_2$ when $X$ is a sphere.
Theorem~\ref{thm:quad} implies that $\Sigma^{-1}[\td{\phi}_k]_2$ is given by the map
$$ \Sigma^{-1} \DD_{2^k}(X)^{\wedge 2}_{h\Sigma_2} \rightarrow \DD_{2^{k+1}}(X) $$
induced by the left action of $\partial_*(\Id)$.  Letting $X = S^1$, and using (\ref{eq:Qbar}), we deduce that
$$ ([\td{\phi}_k]_2)_* Q^j \bar{Q}^{i_1} \cdots \bar{Q}^{i_k}\iota_1 = \sigma \bar{Q}^j \bar{Q}^{i_1} \cdots \bar{Q}^{i_k}. $$
We therefore deduce that
$$ ([\td{\psi}_k]_2)_* Q^j \sigma^k \bar{Q}^{i_1} \cdots \bar{Q}^{i_k} \iota_1 = \sigma^{k+1} \bar{Q}^{j} \bar{Q}^{i_1} \cdots \bar{Q}^{i_k} \iota_1, $$
and the lemma follows.
\end{proof}

Using the above two lemmas, we may now prove Theorem~\ref{thm:psidelta} and deduce Theorem~\ref{thm:main}.

\begin{proof}[Proof of Theorem~\ref{thm:psidelta}]
Endow 
$$ H_* \Omega^\infty \left( \Sigma^k \DD_{2^k}(S^1) \right)^{\wedge 2}_{h\Sigma_2} = \mc{F} H_*\left( \Sigma^k \DD_{2^k}(S^1) \right)^{\wedge 2}_{h\Sigma_2} $$ 
with a weight filtration by defining 
\begin{align*}
w(x) & = 2^{k+1} \quad \text{for $x \in H_*\left( \Sigma^k \DD_{2^k}(S^1) \right)^{\wedge 2}_{h\Sigma_2}$}, \\
w(Q^i x) & = 2 \cdot w(x), \\
w(x \ast y) & = w(x)+w(y).
\end{align*}
Then, by Propositions 2.5 and 2.7 of \cite{Kuhn}, the map
$$ JH_* : H_* \Omega^\infty \Sigma^k \DD_{2^k}(S^1)  \rightarrow H_* \Omega^\infty \left( \Sigma^k \DD_{2^k}(S^1) \right)^{\wedge 2}_{h\Sigma_2} $$ 
preserves the weight filtration.  The maps of the (collapsing) spectral sequences (\ref{eq:SS}) induced by $\alpha_k$ and $[\td{\psi}_k]_2$ imply that the maps
$$
(\Omega^\infty \alpha_k)_*, (\Omega^\infty [\td{\psi}_k]_2)_* :  H_* \Omega^\infty \left( \Sigma^k \DD_{2^k}(S^1) \right)^{\wedge 2}_{h\Sigma_2} \rightarrow H_* \Omega^{\infty}\Sigma^{k+1} \DD_{2^{k+1}}(S^1)
$$
both preserve the weight filtration.  Lemma~\ref{lem:lemma2} implies that on the level of associated graded groups, the maps $E_0 (\Omega^\infty \alpha_k)_*$ and $E_0(\Omega^\infty [\td{\psi}_k]_2)_*$ are equal.  It follows from (\ref{eq:deltafactor}) and Lemma~\ref{lem:lemma2} that 
$$ E_0 (\delta_k)_* = E_0(\psi_k)_*: E_0 H_* B^kD_{2^k}(S^1) \rightarrow E_0 H_* B^{k+1} D_{2^{k+1}} (S^1) $$
as desired.
\end{proof}

\begin{rmk}
The referee points out that the fact that the map $JH_*$ preserves the weight filtration can also be easily deduced from calculus: take the induced map of Goodwillie towers on the natural transformation of functors 
$$ \Sigma^\infty JH: \Sigma^\infty
Q (-) \rightarrow \Sigma^\infty Q (-)^{\wedge 2}_{h\Sigma_2} $$
from spaces to spectra, and apply homology.
\end{rmk}

\begin{proof}[Proof of Theorem~\ref{thm:main}]
Referring to Diagram~(\ref{eq:mydiag}), it is shown in \cite{Kuhn} that
$$ H_* (E_k) = \im (d_k)_* \subseteq H_* B^k D_{2^k}(S^1). $$
The weight filtration on $H_* B^k D_{2^k}(S^1)$ therefore induces a weight filtration on $H_* E_k$.  It follows from Theorem~\ref{thm:psidelta} that 
$$ E_0 (h_k')_* = E_0 (h_k)_*: E_0 H_*E_k \rightarrow E_0 H_* B^{k+1} D_{2^{k+1}}(S^1). $$
Kuhn proved Theorem~\ref{thm:Kuhn} by showing that 
$$ E_0(\td{d}_k)_* \circ E_0(h_k) = \Id: E_0 H_*E_k \rightarrow E_0 H_* E_k. $$
We deduce that 
$$ E_0(\td{d}_k)_* \circ E_0(h'_k) = \Id: E_0 H_*E_k \rightarrow E_0 H_* E_k $$
and thus $\td{d}_k \circ h'_k$ is a self-equivalence of $E_k$.
Consider the induced splittings
$$ B^k D_{2^k}(S^1) \simeq E_{k-1} \times E_{k}. $$
With respect to these splittings, $d_k$ takes ``matrix form''
$$ d_k = 
\begin{bmatrix}
0 & 0 \\
1 & 0
\end{bmatrix}
$$
and there exist self-equivalences $f_k: E_k \rightarrow E_k$ so that
$$
\psi_k =
\begin{bmatrix}
\ast & f_k \\
\ast & 0
\end{bmatrix}.
$$
We deduce that 
$$
d_k \psi_k + \psi_{k-1} d_{k-1} = 
\begin{bmatrix}
f_{k-1} & 0
\\ 
\ast & f_k
\end{bmatrix}
$$
and in particular, $d_k \psi_k + \psi_{k-1} d_{k-1}$ is an equivalence.
\end{proof}

\bibliographystyle{amsalpha}
\nocite{*}
\bibliography{GoodwhiteAGT}

\providecommand{\bysame}{\leavevmode\hbox to3em{\hrulefill}\thinspace}
\providecommand{\MR}{\relax\ifhmode\unskip\space\fi MR }
\providecommand{\MRhref}[2]{%
  \href{http://www.ams.org/mathscinet-getitem?mr=#1}{#2}
}
\providecommand{\href}[2]{#2}
\begin{thebibliography}{BMMS86}

\bibitem[AC]{AroneChing}
Greg Arone and Michael Ching, \emph{Operads and chain rules for calculus of
  functors}, to appear in Ast\'erisque.

\bibitem[AD01]{AroneDwyer}
G.~Z. Arone and W.~G. Dwyer, \emph{Partition complexes, {T}its buildings and
  symmetric products}, Proc. London Math. Soc. (3) \textbf{82} (2001), no.~1,
  229--256.

\bibitem[ADL08]{AroneDwyerLesh}
Gregory~Z. Arone, William~G. Dwyer, and Kathryn Lesh, \emph{Loop structures in
  {T}aylor towers}, Algebr. Geom. Topol. \textbf{8} (2008), no.~1, 173--210.

\bibitem[AL10]{AroneLesh}
Gregory~Z. Arone and Kathryn Lesh, \emph{Augmented {$\Gamma$}-spaces, the
  stable rank filtration, and a {$bu$} analogue of the {W}hitehead conjecture},
  Fund. Math. \textbf{207} (2010), no.~1, 29--70.

\bibitem[AM99]{AroneMahowald}
Greg Arone and Mark Mahowald, \emph{The {G}oodwillie tower of the identity
  functor and the unstable periodic homotopy of spheres}, Invent. Math.
  \textbf{135} (1999), no.~3, 743--788.

\bibitem[Aro98]{Arone}
Greg Arone, \emph{Iterates of the suspension map and {M}itchell's finite
  spectra with {$A_k$}-free cohomology}, Math. Res. Lett. \textbf{5} (1998),
  no.~4, 485--496.

\bibitem[Beh]{Behrens}
Mark Behrens, \emph{The {G}oodwillie tower and the {EHP} sequence}, to appear
  in Mem. AMS.

\bibitem[BMMS86]{Hinfty}
R.~R. Bruner, J.~P. May, J.~E. McClure, and M.~Steinberger, \emph{{$H_\infty $}
  ring spectra and their applications}, Lecture Notes in Mathematics, vol.
  1176, Springer-Verlag, Berlin, 1986.

\bibitem[Chi05]{Ching}
Michael Ching, \emph{Bar constructions for topological operads and the
  {G}oodwillie derivatives of the identity}, Geom. Topol. \textbf{9} (2005),
  833--933 (electronic).

\bibitem[CLM76]{CohenLadaMay}
Frederick~R. Cohen, Thomas~J. Lada, and J.~Peter May, \emph{The homology of
  iterated loop spaces}, Lecture Notes in Mathematics, Vol. 533,
  Springer-Verlag, Berlin, 1976.

\bibitem[Goo03]{Goodwillie}
Thomas~G. Goodwillie, \emph{Calculus. {III}. {T}aylor series}, Geom. Topol.
  \textbf{7} (2003), 645--711 (electronic).

\bibitem[KMP82]{KMP}
N.~J. Kuhn, S.~A. Mitchell, and S.~B. Priddy, \emph{The {W}hitehead conjecture
  and splitting {$B({\bf Z}/2)^{k}$}}, Bull. Amer. Math. Soc. (N.S.) \textbf{7}
  (1982), no.~1, 255--258.

\bibitem[Kuh82]{Kuhn}
Nicholas~J. Kuhn, \emph{A {K}ahn-{P}riddy sequence and a conjecture of {G}.
  {W}. {W}hitehead}, Math. Proc. Cambridge Philos. Soc. \textbf{92} (1982),
  no.~3, 467--483.

\bibitem[Kuh07]{KuhnChrom}
\bysame, \emph{Goodwillie towers and chromatic homotopy: an overview},
  Proceedings of the {N}ishida {F}est ({K}inosaki 2003), Geom. Topol. Monogr.,
  vol.~10, Geom. Topol. Publ., Coventry, 2007, pp.~245--279. \MR{2402789
  (2009h:55009)}

\bibitem[Nak58]{Nakaoka}
Minoru Nakaoka, \emph{Cohomology mod {$p$} of symmetric products of spheres},
  J. Inst. Polytech. Osaka City Univ. Ser. A \textbf{9} (1958), 1--18.
  \MR{0121788 (22 \#12519b)}

\bibitem[Nis87]{Nishida}
Goro Nishida, \emph{On the spectra {$L(n)$} and a theorem of {K}uhn}, Homotopy
  theory and related topics ({K}yoto, 1984), Adv. Stud. Pure Math., vol.~9,
  North-Holland, Amsterdam, 1987, pp.~273--286.

\bibitem[Rec70]{Rector}
David~L. Rector, \emph{Steenrod operations in the {E}ilenberg-{M}oore spectral
  sequence}, Comment. Math. Helv. \textbf{45} (1970), 540--552.

\bibitem[Wei95]{Weiss}
Michael Weiss, \emph{Orthogonal calculus}, Trans. Amer. Math. Soc. \textbf{347}
  (1995), no.~10, 3743--3796.

\end{thebibliography}

\end{document}